\documentclass{amsart}
\usepackage{setspace}
\usepackage{a4}
\usepackage{amsthm}
\usepackage{latexsym}
\usepackage{amsfonts}
\usepackage{graphicx}
\usepackage{textcomp}
\usepackage{cite}
\usepackage{enumerate}
\usepackage{amssymb}
\usepackage{hyperref}
\usepackage{amsmath}
\usepackage{tikz}
\usepackage[mathscr]{euscript}
\usepackage{mathtools}
\newtheorem{theorem}{Theorem}[section]

\newtheorem{definition}[theorem]{Definition}
\newtheorem{example}[theorem]{Example}

\newtheorem{question}[theorem]{Question}
\title{This is the title}
\usepackage{fancyhdr}

\begin{document}
	\vspace{0.9cm}	
\hrule\hrule\hrule\hrule\hrule
\vspace{0.3cm}	
\begin{center}
{\bf{p-adic Ghobber-Jaming Uncertainty Principle}}\\
\vspace{0.3cm}
\hrule\hrule\hrule
\vspace{0.3cm}
\textbf{K. Mahesh Krishna}\\
School of Mathematics and Natural Sciences\\
Chanakya University Global Campus\\
NH-648, Haraluru Village\\
Devanahalli Taluk, 	Bengaluru  North District\\
Karnataka State 562 110 India \\
Email: kmaheshak@gmail.com\\

Date: \today
\hrule\hrule
\end{center}

%\begin{center}
%	Dedicated to Prof. Chaitanya G. K. 
%\end{center}
\hrule\hrule
\vspace{0.5cm}
%--------------------------------------
\textbf{Abstract}: Let $\{\tau_j\}_{j=1}^n$ and $\{\omega_k\}_{k=1}^n$ be  two orthonormal bases  for a finite dimensional p-adic Hilbert space $\mathcal{X}$. Let  $M,N\subseteq \{1, \dots, n\}$ be such that 
\begin{align*}
\displaystyle \max_{j \in M, k \in N}|\langle \tau_j, \omega_k \rangle|<1,
\end{align*}
where  $o(M)$  is the cardinality of  $M$. Then for all $x \in \mathcal{X}$,  we show that 
\begin{align}\label{FGJU}
	\|x\|\leq  \left(\frac{1}{1-\displaystyle \max_{j \in M, k \in N}|\langle \tau_j, \omega_k \rangle|}\right)\max\left\{\displaystyle \max_{j \in M^c}|\langle x, \tau_j\rangle |, \displaystyle \max_{k \in N^c}|\langle x, \omega_k\rangle |\right\}.
\end{align}
 We call Inequality (\ref{FGJU}) as \textbf{p-adic Ghobber-Jaming Uncertainty Principle}. Inequality (\ref{FGJU}) is the p-adic version of uncertainty principle obtained by Ghobber and Jaming \textit{[Linear Algebra Appl., 2011]}. We also derive analogues of Inequality (\ref{FGJU}) for non-Archimedean Banach spaces.



\hrule

%\tableofcontents
\hrule
\section{Introduction}
Let $d \in \mathbb{N}$ and  $~\widehat{}:\mathcal{L}^2 (\mathbb{R}^d) \to \mathcal{L}^2 (\mathbb{R}^d)$ be the unitary Fourier transform obtained by extending uniquely the bounded linear operator 
\begin{align*}
\widehat{}:\mathcal{L}^1 (\mathbb{R}^d)\cap  \mathcal{L}^2	 (\mathbb{R}^d) \ni f \mapsto \widehat{f} \in  C_0(\mathbb{R}^d); \quad \widehat{f}: \mathbb{R}^d \ni \xi \mapsto \widehat{f}(\xi)\coloneqq \int_{\mathbb{R}^d}	f(x)e^{-2\pi i  \langle x, \xi \rangle}\,dx\ \in \mathbb{C}.
\end{align*}
In 2007, Jaming \cite{JAMING} extended the uncertainty principle obtained by Nazarov for $\mathbb{R}$ in 1993 \cite{NAZAROV} (cf. \cite{HARVINJORICKE}). In the following theorem, Lebesgue measure on $\mathbb{R}^d$ is denoted by $m$. Mean width of a measurable subset  $E$ of $\mathbb{R}^d$ having finite measure  is denoted by $w(E)$. Complement of  $E\subseteq\mathbb{R}^d$ is denoted by $E^c$.
\begin{theorem} \cite{JAMING, NAZAROV}\label{JN} (\textbf{Nazarov-Jaming Uncertainty Principle}) For each $d \in \mathbb{N}$, there exists a universal constant $C_d$ (depends upon $d$) satisfying the following: If $E, F \subseteq \mathbb{R}^d$ are measurable subsets having finite measure, then for all $f \in \mathcal{L}^2 (\mathbb{R}^d)$, 
	\begin{align*}
\int_{\mathbb{R}^d}	|f(x)|^2\,dx &\leq C_d e^{C_d \min \{m(E)m(F), m(E)^\frac{1}{d}w(F), m(F)^\frac{1}{d}w(E)\}} 
&\left[\int_{E^c}	|f(x)|^2\,dx+\int_{F^c}	|\widehat{f}(\xi)|^2\,d\xi\right].
	\end{align*} 
In particular, if $f$ is supported on $E$ and 	$\widehat{f}$ is supported on $F$, then $f=0$. 
\end{theorem}
In 2011, Ghobber and Jaming derived the following finite dimensional version of 
Theorem \ref{JN}.  Given a subset  $M\subseteq \{1, \dots, n\}$, the number of elements in $M$ is denoted by $o(M)$. Complement of $M\subseteq \{1, \dots, n\}$ is denoted by $M^c$. 
\begin{theorem}\cite{GHOBBERJAMING}\label{GJ} (\textbf{Ghobber-Jaming  Uncertainty Principle})
Let 	$\{\tau_j\}_{j=1}^n$ and $\{\omega_j\}_{j=1}^n$ be orthonormal bases for   the Hilbert space $\mathbb{C}^n$.  If $M,N\subseteq \{1, \dots, n\}$ are such that 
\begin{align*}
o(M)o(N)< \frac{1}{\displaystyle \max_{1\leq j, k\leq n}|\langle\tau_j, \omega_k\rangle |^2},
\end{align*}
then for all $h \in \mathbb{C}^n$, 
\begin{align*}
	\|h\|\leq \left(1+\frac{1}{1-\sqrt{o(M)o(N)}\displaystyle\max_{1\leq j,k\leq n}|\langle\tau_j, \omega_k\rangle |}\right)\left[\left(\sum_{j\in M^c}|\langle h, \tau_j\rangle |^2\right)^\frac{1}{2}+\left(\sum_{k\in N^c}|\langle h, \omega_k\rangle |^2\right)^\frac{1}{2}\right].
\end{align*}
In particular, if $h$ is supported on $M$ in the expansion using basis $\{\tau_j\}_{j=1}^n$ and $h$ is supported on $N$ in the expansion using basis $\{\omega_j\}_{j=1}^n$, then $h=0$. 
\end{theorem}
Recently, Theorem \ref{GJ} has been extended to Banach spaces using the notion of p-orthonormal bases.
\begin{theorem} \label{K} \cite{KRISHNA} (\textbf{Functional Ghobber-Jaming  Uncertainty Principle})
	Let $(\{f_j\}_{j=1}^n, \{\tau_j\}_{j=1}^n)$	and $(\{g_k\}_{k=1}^n, \{\omega_k\}_{k=1}^n)$ be p-orthonormal bases  for a finite dimensional Banach space $\mathcal{X}$.  If $M,N\subseteq \{1, \dots, n\}$ are such that 
	\begin{align*}
		o(M)^\frac{1}{q}o(N)^\frac{1}{p}< \frac{1}{\displaystyle \max_{1\leq j,k\leq n}|g_k(\tau_j) |},
	\end{align*}
	then for all $x \in \mathcal{X}$, 	
	\begin{align*}
		\|x\|\leq \left(1+\frac{1}{1-o(M)^\frac{1}{q}o(N)^\frac{1}{p}\displaystyle\max_{1\leq j,k\leq n}|g_k(\tau_j)|}\right)\left[\left(\sum_{j\in M^c}|f_j(x)|^p\right)^\frac{1}{p}+\left(\sum_{k\in N^c}|g_k(x) |^p\right)^\frac{1}{p}\right].
	\end{align*}
	In particular, if $x$ is supported on $M$ in the expansion using basis $\{\tau_j\}_{j=1}^n$ and $x$ is supported on $N$ in the expansion using basis $\{\omega_k\}_{k=1}^n$, then $x=0$. 
\end{theorem}
 It is  reasonable to ask whether there  are p-adic (non-Archimedean) versions of Theorems \ref{GJ} and \ref{K}? We are going to answer this question in the paper.

\section{p-adic Ghobber-Jaming Uncertainty Principle}

We begin by recalling the notion of p-adic Hilbert space.  Kalisch made important study of p-adic Hilbert spaces in 1947 by imposing certain conditions on base field \cite{KALISCH}. We are going to use much more general notion which recently attracted lot of attention, see \cite{KRISHNA3, KRISHNA4, ANIELLOMANCINIPARISI, ANIELLOMANCINIPARISI2}.
\begin{definition}\cite{KRISHNA3, KRISHNA4, ANIELLOMANCINIPARISI, ANIELLOMANCINIPARISI2} \label{PADICDEF}
	Let $\mathbb{K}$ be a non-Archimedean  valued field with valuation $|\cdot|$ and $\mathcal{X}$ be a non-Archimedean Banach space with norm $\|\cdot\|$ over $\mathbb{K}$. We say that $\mathcal{X}$ is a \textbf{p-adic Hilbert space} if there is a map (called as p-adic inner product) $\langle \cdot, \cdot \rangle: \mathcal{X} \times \mathcal{X} \to \mathbb{K}$ satisfying following.
	\begin{enumerate}[\upshape (i)]
		\item If $x \in \mathcal{X}$ is such that $\langle x,y \rangle =0$ for all $y \in \mathcal{X}$, then $x=0$.
		\item $\langle x, y \rangle =\langle y, x \rangle$ for all $x,y \in \mathcal{X}$.
		\item $\langle \alpha x+y,z \rangle =\alpha \langle x,  y \rangle+\langle y,z\rangle$ for all  $\alpha  \in \mathbb{K}$, for all $x,y,z \in \mathcal{X}$.
		\item $|\langle x, y \rangle |\leq \|x\|\|y\|$ for all $x,y \in \mathcal{X}$.
	\end{enumerate}
\end{definition}
The following is the  standard example of a p-adic Hilbert space.
\begin{example}\cite{KALISCH} \label{E}
	Let $p$ be a prime. For $d \in \mathbb{N}$, let $\mathbb{Q}_p^d$ be the standard p-adic Hilbert space equipped with the inner product 
	\begin{align*}
		\langle (a_j)_{j=1}^d,(b_j)_{j=1}^d\rangle := \sum_{j=1}^da_jb_j,  \quad \forall (a_j)_{j=1}^d,(b_j)_{j=1}^d \in \mathbb{Q}_p^d
	\end{align*}
	and the norm 
	\begin{align*}
		\|(x_j)_{j=1}^d\|:= \max_{1\leq j \leq d}|x_j|, \quad \forall (x_j)_{j=1}^d\in 	\mathbb{Q}_p^d.
	\end{align*}
\end{example}
Note that we can replace $\mathbb{Q}_p$ in Example \ref{E} by any non-Archimedean valued field $\mathbb{K}$. 
In the paper, we use the following notion of orthonormal basis for p-adic Hilbert spaces.
\begin{definition}\label{PONB}
	Let $\mathcal{X}$  be a  p-adic Hilbert  space over $\mathbb{K}$.   A basis  $\{\tau_j\}_{j=1}^n$ for  $\mathcal{X}$ is said to be an  orthonormal basis for    $\mathcal{X}$   if  the following conditions hold.	
		\begin{enumerate}[\upshape(i)]
		\item $\|\tau_j\|\leq 1$ for all $1\leq j\leq n$.
		\item $	\langle \tau_j, \tau_k\rangle =\delta_{j,k}$ for all $ 1\leq j, k \leq n$.
	\end{enumerate}
\end{definition}
 Note that given an orthonormal basis $\{\tau_j\}_{j=1}^n$ for  $\mathcal{X}$,  condition (ii) in Definition \ref{PONB} says that 
 \begin{align*}
 	1=|\langle \tau_j, \tau_j \rangle |\leq \|\tau_j\|^2, \quad \forall 1\leq j \leq n.
 \end{align*}
 By using condition (i) in Definition \ref{PONB} we then get $\|\tau_j\|= 1$ for all $1\leq j\leq n$.
Given an orthonormal basis $\{\tau_j\}_{j=1}^n$ for $\mathcal{X}$, we have following.
	\begin{enumerate}[\upshape(i)]
	\item For every $x \in \mathcal{X}$, 
	\begin{align*}
	\text{(p-adic Fourier expansion)}\quad \quad \quad 	x=\sum_{j=1}^{n}\langle x, \tau_j\rangle  \tau_j.
	\end{align*}
\item For every $x \in \mathcal{X}$, 
	\begin{align*}
	\text{(p-adic Parseval formula for norm)}\quad \quad \quad	\|x\|=\displaystyle \max_{1\leq j\leq n}|\langle x, \tau_j\rangle|.
\end{align*}
\item For every $x, y \in \mathcal{X}$, 
\begin{align*}
	\text{(p-adic Parseval formula for inner product)}\quad \quad \quad	\langle x, y \rangle=\sum_{j=1}^{n}\langle x, \tau_j\rangle\langle \tau_j, y\rangle.
\end{align*}
	\item For every $(a_j)_{j=1}^n \in \mathbb{K}^n$, 
	\begin{align*}
		\left\|\sum_{j=1}^na_j\tau_j \right\|=\displaystyle \max_{1\leq j\leq n}|a_j|.
	\end{align*}
\end{enumerate}
Let $\mathcal{X}$ be a p-adic Hilbert space. Recall that a linear operator $T:\mathcal{X} \to \mathcal{X}$ is said to be an isometry if $\|Tx\|=\|x\|$ for all $x \in \mathcal{X}$. An invertible isometry $T:\mathcal{X} \to \mathcal{X}$ is said to be unitary if $\langle Tx, Ty\rangle =\langle x, y \rangle $ for all $x, y \in \mathcal{X}$. Note that unlike real or complex Hilbert spaces, neither 
\begin{align*}
\langle Tx, Ty\rangle =\langle x, y \rangle, ~ \forall x, y \in \mathcal{X}	\quad \Rightarrow \|Tx\|=\|x\|, ~ \forall x  \in \mathcal{X}
\end{align*}
nor 
\begin{align*}
\|Tx\|=\|x\|, ~ \forall x  \in \mathcal{X}	\quad \Rightarrow \langle Tx, Ty\rangle =\langle x, y \rangle, ~ \forall x, y \in \mathcal{X}.
\end{align*}
In the next result we characterize all orthonormal bases for a given p-adic Hilbert space.
\begin{theorem}
	Let $\{\tau_j\}_{j=1}^n$ be an orthonormal basis for $\mathcal{X}$. Then a collection 	$\{\omega_j\}_{j=1}^n$ in $\mathcal{X}$ is an orthonormal basis for $\mathcal{X}$ if and only if there is a unitary $V:\mathcal{X} \to \mathcal{X}$ such that 
	\begin{align*}
		 \omega_j=V\tau_j, \quad \forall 1\leq j \leq n.
	\end{align*}
\end{theorem}
\begin{proof}
	$(\Rightarrow)$ Define $V:\mathcal{X}\ni x \mapsto \sum_{j=1}^{n}\langle x, \tau_j \rangle \omega_j \in \mathcal{X}$. Since $\{\omega_j\}_{j=1}^n$ is a basis for $\mathcal{X}$, $V$ is invertible with inverse $V^{-1}:\mathcal{X}\ni x \mapsto \sum_{j=1}^{n}\langle x, \omega_j \rangle\tau_j \in \mathcal{X}$. For $x \in \mathcal{X}$, 
	
	\begin{align*}
		\|Vx\|=\left\| \sum_{j=1}^{n}\langle x, \tau_j \rangle\omega_j \right\|=\displaystyle \max_{1\leq j\leq n}|\langle x, \tau_j\rangle|=\|x\|.
	\end{align*}
	Therefore $V$ is an isometry. For $x, y \in \mathcal{X}$, 
	\begin{align*}
		\langle Vx, Vy\rangle &=\left \langle \sum_{j=1}^{n}\langle x, \tau_j \rangle\omega_j, \sum_{k=1}^{n}\langle y, \tau_k \rangle\omega_k\right \rangle=\sum_{j=1}^{n}\langle x, \tau_j \rangle \langle \tau_j, y \rangle =\langle x, y \rangle.
	\end{align*}
Therefore $V$ is a unitary. We clearly have $\omega_j=V\tau_j,  \forall 1\leq j \leq n.$ \\
	$(\Leftarrow)$ Since $V$ is invertible, $\{\omega_j\}_{j=1}^n$ is a basis for $\mathcal{X}$.   Since $V$ is an isometry, we have $\|\omega_j\|=\|\tau_j\|\leq 1$ for all $1\leq j \leq n$. Since $V$ is unitary,  we have 
	\begin{align*}
			\langle \omega_j, \omega_k\rangle=\langle V\tau_j, V\tau_k\rangle=\langle \tau_j, \tau_k\rangle =\delta_{j,k}, \quad \forall 1\leq j,k \leq n.
	\end{align*}
\end{proof}
Now we derive p-adic version of Theorem  \ref{GJ}.
\begin{theorem}(\textbf{p-adic Ghobber-Jaming  Uncertainty Principle})
	Let $\mathcal{X}$  be a finite dimensional  p-adic Hilbert  space over $\mathbb{K}$. 	Let $\{\tau_j\}_{j=1}^n$ and $\{\omega_k\}_{k=1}^n$ be  two orthonormal bases  for  $\mathcal{X}$. If $M,N\subseteq \{1, \dots, n\}$ are such that 
	\begin{align*}
	\displaystyle \max_{j \in M, k \in N}|\langle \tau_j, \omega_k \rangle|<1,
	\end{align*}
then for all $x \in \mathcal{X}$, 	
\begin{align}\label{PGJ}
		\|x\|\leq  \left(\frac{1}{1-\displaystyle \max_{j \in M, k \in N}|\langle \tau_j, \omega_k \rangle|}\right)\max\left\{\displaystyle \max_{j \in M^c}|\langle x, \tau_j\rangle |, \displaystyle \max_{k \in N^c}|\langle x, \omega_k\rangle |\right\}.
\end{align}
In particular, if $x$ is supported on $M$ in the expansion using basis $\{\tau_j\}_{j=1}^n$ and $x$ is supported on $N$ in the expansion using basis $\{\omega_j\}_{j=1}^n$, then $x=0$. 
\end{theorem}
\begin{proof}
Given $S\subseteq \{1, \dots, n\}$, define 
\begin{align*}
	&P_Sx\coloneqq \sum_{j\in S}\langle x, \tau_j\rangle\tau_j, \quad \forall x \in \mathcal{X},\\
	&\|x\|_{S, \tau}\coloneqq \displaystyle \max_{j \in S}|\langle x, \tau_j\rangle|, \quad \forall x \in \mathcal{X},\\
	&\|x\|_{S, \omega}\coloneqq \displaystyle \max_{k \in S}|\langle x, \omega_k\rangle|, \quad \forall x \in \mathcal{X}.
\end{align*}
Then 
\begin{align*}
	\|P_Sx\|=\left\|\sum_{j\in S}\langle x, \tau_j \rangle \tau_j\right\|=\displaystyle \max_{j \in S}|\langle x, \tau_j\rangle|=\|x\|_{S, \tau}, \quad \forall x \in \mathcal{X}.
\end{align*}
Define $V:\mathcal{X}\ni x \mapsto \sum_{k=1}^{n}\langle x, \omega_k\rangle\tau_k \in \mathcal{X}$. Then $V$ is an invertible isometry. Using $V$ we have
\begin{align*}
	\|P_SVx\|&=\left\|\sum_{j\in S}\langle Vx, \tau_j\rangle \tau_j\right\|=\displaystyle \max_{j \in S}|\langle Vx, \tau_j\rangle|=\displaystyle \max_{j \in S}\left|\left\langle \sum_{k=1}^{n}\langle x, \omega_k\rangle\tau_k, \tau_j\right\rangle\right|\\
	&=\displaystyle \max_{j \in S}|\langle x, \omega_j\rangle|=\|x\|_{S, \omega}, \quad \forall x \in \mathcal{X}.
\end{align*}
For $x \in \mathcal{X}$, we  now find 

\begin{align*}
	\|P_NVP_Mx\|&=\left\|\sum_{k\in N}\langle VP_Mx, \tau_k\rangle \tau_k\right\|=\displaystyle \max_{k \in N}|\langle VP_Mx, \tau_k\rangle|\\
	&=\displaystyle \max_{k \in N}\left|\left\langle V\left(\sum_{j \in M}\langle x, \tau_j\rangle \tau_j\right), \tau_k\right\rangle\right|=\displaystyle \max_{k \in N}\left|\left\langle \sum_{j \in M}\langle x, \tau_j\rangle V\tau_j, \tau_k\right\rangle\right|\\
	&=\displaystyle \max_{k \in N}\left| \sum_{j \in M}\langle x, \tau_j\rangle \langle V\tau_j, \tau_k\rangle\right|=\displaystyle \max_{k \in N}\left| \sum_{j \in M}\langle x, \tau_j\rangle \left\langle \sum_{r=1}^{n}\langle \tau_j, \omega_r\rangle \tau_r, \tau_k\right\rangle\right|\\
	&=\displaystyle \max_{k \in N}\left| \sum_{j \in M}\langle x, \tau_j\rangle \left\langle \tau_j, \omega_k\right\rangle\right|\leq \displaystyle \max_{k \in N}\max_{j \in M}|\langle x, \tau_j\rangle \left\langle \tau_j, \omega_k\right\rangle|\\
	&\leq \left(\displaystyle \max_{j \in M, k \in N}|\langle \tau_j, \omega_k \rangle|\right) \max_{j \in M}|\langle x, \tau_j\rangle|\leq \left(\displaystyle \max_{j \in M, k \in N}|\langle \tau_j, \omega_k \rangle|\right)\|x\|.
\end{align*}
Therefore 
\begin{align}\label{PVP}
	\|P_NVP_M\|\leq \displaystyle \max_{j \in M, k \in N}|\langle \tau_j, \omega_k \rangle|.
\end{align}
Now let $y \in \mathcal{X}$ be such that $\{j \in \{1, \dots, n\}: \langle y, \tau_j \rangle \neq 0\} \subseteq M.$
Then 
\begin{align*}
\|P_NVy\|=\|P_NVP_My\|\leq \|P_NVP_M\|\|y\|	
\end{align*}
 and 
\begin{align*}
	\|y\|_{N^c, \omega}&=\|P_{N^c}Vy\|=\|Vy-P_NVy\|\\
	&\geq \|Vy\|-\|P_NVy\|=\|y\|-\|P_NVy\| \\
	&\geq \|y\|-\|P_NVP_M\|\|y\|.
\end{align*}
Therefore 
\begin{align}\label{RT}
	\|y\|_{N^c, \omega}\geq (1-\|P_NVP_M\|)\|y\|.
\end{align}
Let $x \in \mathcal{X}$.   Note that $P_Mx$ satisfies $	\{j \in \{1, \dots, n\}: \langle P_Mx, \tau_j\rangle \neq 0\} \subseteq M.$ Now using (\ref{PVP}) and  (\ref{RT}) we get 

\begin{align*}
	\|x\|&=\|P_{M^c}x+P_Mx\|\leq \max\{\|P_{M^c}x\|, \|P_Mx\|\}\\
	&\leq \max \left\{\|P_{M^c}x\|, \frac{1}{1-\|P_NVP_M\|}	\|P_Mx\|_{N^c, \omega}\right\} \\
	&= \max \left\{\|P_{M^c}x\|, \frac{1}{1-\|P_NVP_M\|}	\|P_{N^c}VP_Mx\|\right\} \\
	&=\max \left\{\|P_{M^c}x\|, \frac{1}{1-\|P_NVP_M\|}	\|P_{N^c}V(x-P_{M^c}x)\|\right\} \\
	&\leq \max \left\{\|P_{M^c}x\|, \frac{1}{1-\|P_NVP_M\|}	\|P_{N^c}Vx\|, \frac{1}{1-\|P_NVP_M\|}	\|P_{N^c}VP_{M^c}x\|\right\} \\
	&\leq \max \left\{\|P_{M^c}x\|, \frac{1}{1-\|P_NVP_M\|}	\|P_{N^c}Vx\|, \frac{1}{1-\|P_NVP_M\|}	\|P_{M^c}x\|\right\} \\
	&=\max \left\{\frac{1}{1-\|P_NVP_M\|}	\|P_{M^c}x\|, \frac{1}{1-\|P_NVP_M\|}	\|P_{N^c}Vx\|\right\} \\
	&=\left(\frac{1}{1-\|P_NVP_M\|}\right)\max \left\{\|P_{M^c}x\|, \|P_{N^c}Vx\|\right\} \\
 &=\left(\frac{1}{1-\|P_NVP_M\|}\right)\max \left\{\|x\|_{M^c, \tau}, \|x\|_{N^c, \omega}\right\} \\
 &=\left(\frac{1}{1-\|P_NVP_M\|}\right) \max\left\{\displaystyle \max_{j \in M^c}|\langle x, \tau_j\rangle |, \displaystyle \max_{k \in N^c}|\langle x, \omega_k\rangle |\right\}\\
 &\leq \left(\frac{1}{1-\displaystyle \max_{j \in M, k \in N}|\langle \tau_j, \omega_k \rangle|}\right)\max\left\{\displaystyle \max_{j \in M^c}|\langle x, \tau_j\rangle |, \displaystyle \max_{k \in N^c}|\langle x, \omega_k\rangle |\right\}.
\end{align*}
\end{proof}

\section{Non-archimedean  Ghobber-Jaming Uncertainty Principle}
We now want to derive Ghobber-Jaming uncertainty principle for non-Archimedean Banach spaces. For this, we introduce the following notion of orthonormal bases. 

\begin{definition}\label{FONB}
	Let $\mathcal{X}$  be a  finite dimensional non-Archimedean Banach  space over $\mathbb{K}$. Let $\{\tau_j\}_{j=1}^n$ be a basis for   $\mathcal{X}$ and 	let $\{f_j\}_{j=1}^n$ be the coordinate functionals associated with $\{\tau_j\}_{j=1}^n$. The pair $(\{f_j\}_{j=1}^n, \{\tau_j\}_{j=1}^n)$ is said to be an \textbf{orthonormal basis}  for $\mathcal{X}$ if  the following conditions hold.	
		\begin{enumerate}[\upshape(i)]
		\item $\|\tau_j\|\leq 1$ for all $1\leq j\leq n$.
		\item $\|f_j\|\leq 1$ for all $1\leq j\leq n$.
	\end{enumerate}
\end{definition}
Let $(\{f_j\}_{j=1}^n, \{\tau_j\}_{j=1}^n)$ be an orthonormal basis for $\mathcal{X}$. Definition \ref{FONB} then gives
\begin{align*}
	1=|f_j(\tau_j)|\leq \|f_j\|\|\tau_j\|\leq1, \quad \forall 1\leq j \leq n.
\end{align*}
Hence $\|f_j\|=\|\tau_j\|=1$ for all $1\leq j \leq n$.
\begin{example}
	Let $\mathbb{K}$ be any non-Archimedean valued field. Given $d \in \mathbb{N}$, we equip $\mathbb{K}^d$ with the norm 
	\begin{align*}
		\|(a_j)_{j=1}^d\|:= \max_{1\leq j \leq d}|a_j|, \quad \forall (a_j)_{j=1}^d\in 	\mathbb{K}^d.	
	\end{align*}
Then $\mathbb{K}^d$  is a $d$-dimensional non-Archimedean Banach  space over $\mathbb{K}$. Let $\{e_j\}_{j=1}^d$ be the  canonical basis for   $\mathbb{K}^d$ and $\{\zeta_j\}_{j=1}^d$ be the coordinate functionals associated with $\{e_j\}_{j=1}^d$. Then $\|\zeta_j\|= \|e_j\|=1$ for all $1\leq j\leq d$. So $(\{\zeta_j\}_{j=1}^d, \{e_j\}_{j=1}^d)$ is an  orthonormal basis  for $\mathbb{K}^d$.
\end{example}
Given an orthonormal basis $(\{f_j\}_{j=1}^n, \{\tau_j\}_{j=1}^n)$ for $\mathcal{X}$, we have following.
\begin{enumerate}[\upshape(i)]
	\item For every $x \in \mathcal{X}$, 
	\begin{align*}
		\text{(Non-Archimedean basis expansion)}\quad \quad \quad 	x=\sum_{j=1}^{n}f_j(x)  \tau_j.
	\end{align*}
	\item For every $x \in \mathcal{X}$, 
	\begin{align*}
		\text{(Non-Archimedean Parseval formula for norm)}\quad \quad \quad	\|x\|=\displaystyle \max_{1\leq j\leq n}|f_j(x)|.
	\end{align*}
\item For every linear functional $\phi: \mathcal{X}\to \mathbb{K}$, 
\begin{align*}
\phi=\sum_{j=1}^{n}\phi(\tau_j)f_j.	
\end{align*}
\item For every linear functional $\phi: \mathcal{X}\to \mathbb{K}$, 
\begin{align*}
	\|\phi\|=\displaystyle \max_{1\leq j\leq n}|\phi(\tau_j)|.
\end{align*}
	\item For every $(a_j)_{j=1}^n, (b_j)_{j=1}^n \in \mathbb{K}^n$, 
	\begin{align*}
		\left\|\sum_{j=1}^na_j\tau_j \right\|=\displaystyle \max_{1\leq j\leq n}|a_j| \quad \text{and} \quad 	\left\|\sum_{j=1}^nb_jf_j \right\|=\displaystyle \max_{1\leq j\leq n}|b_j|.
	\end{align*}
\end{enumerate}
In the next result we characterize all orthonormal bases for a given non-Archimedean Banach  space.
\begin{theorem}
	Let $(\{f_j\}_{j=1}^n, \{\tau_j\}_{j=1}^n)$ be an orthonormal basis for $\mathcal{X}$. Then a pair 	$(\{g_j\}_{j=1}^n, \{\omega_j\}_{j=1}^n)$ is an orthonormal basis for $\mathcal{X}$ if and only if there is an invertible linear isometry $V:\mathcal{X} \to \mathcal{X}$ such that 
	\begin{align*}
		g_j=f_jV^{-1}, ~ \omega_j=V\tau_j, \quad \forall 1\leq j \leq n.
	\end{align*}
\end{theorem}
\begin{proof}
	$(\Rightarrow)$ Define $V:\mathcal{X}\ni x \mapsto \sum_{j=1}^{n}f_j(x)\omega_j \in \mathcal{X}$. Since $\{\omega_j\}_{j=1}^n$ is a basis for $\mathcal{X}$, $V$ is invertible with inverse $V^{-1}:\mathcal{X}\ni x \mapsto \sum_{j=1}^{n}g_j(x)\tau_j \in \mathcal{X}$. For $x \in \mathcal{X}$, 
	\begin{align*}
		\|Vx\|=\left\| \sum_{j=1}^{n}f_j(x)\omega_j \right\|=\max_{1\leq j\leq n}|f_j(x)|=\|x\|.
	\end{align*}
	Therefore $V$ is an  isometry. Note that we clearly have $\omega_j=V\tau_j,  \forall 1\leq j \leq n.$ Now let $1\leq j \leq n$. Then 
	\begin{align*}
		f_j(V^{-1}x)=f_j \left(\sum_{k=1}^{n}g_k(x)\tau_k \right)=\sum_{k=1}^{n}g_k(x)f_j(\tau_k)=g_j(x), \quad \forall x \in \mathcal{X}.
	\end{align*}
	$(\Leftarrow)$ Since $V$ is invertible, $\{\omega_j\}_{j=1}^n$ is a basis for $\mathcal{X}$. Now we see that $g_j(\omega_k)=f_j(V^{-1}V\tau_k)=f_j(\tau_k)=\delta_{j,k}$ for all $1\leq j, k \leq n$. Therefore $\{g_j\}_{j=1}^n$ is the coordinate functionals associated with $\{\omega_j\}_{j=1}^n$.  Since $V$ is an isometry, we have $\|\omega_j\|=\|\tau_j\|\leq 1$ for all $1\leq j \leq n$. Since $V$ is also  invertible, we have 
	\begin{align*}
		\|g_j\|&=\displaystyle\sup_{x\in \mathcal{X}\setminus \{0\}}\frac{|g_j(x)|}{\|x\|}=\displaystyle\sup_{x\in \mathcal{X}\setminus \{0\}}\frac{|f_j(V^{-1}x)|}{\|x\|}=\displaystyle\sup_{Vy \in \mathcal{X}\setminus \{0\}}\frac{|f_j(y)|}{\|Vy\|}\\
		&=\displaystyle\sup_{y \in \mathcal{X}\setminus \{0\}}\frac{|f_j(y)|}{\|y\|}=\|f_j\|\leq 1, \quad \forall 1\leq j \leq n.
	\end{align*}
\end{proof}

We now derive non-Archimedean version of Theorem \ref{K}.
\begin{theorem}\label{NGJ}
	(\textbf{Non-Archimedean  Ghobber-Jaming  Uncertainty Principle})
	Let $\mathcal{X}$  be a finite dimensional non-Archimedean Banach  space over $\mathbb{K}$.	Let $(\{f_j\}_{j=1}^n, \{\tau_j\}_{j=1}^n)$	and $(\{g_k\}_{k=1}^n, \{\omega_k\}_{k=1}^n)$ be orthonormal bases  for $\mathcal{X}$.  If $M,N\subseteq \{1, \dots, n\}$ are such that 
	\begin{align*}
		\displaystyle \max_{j \in M, k \in N}|g_k(\tau_j)|<1,
	\end{align*}
	then for all $x \in \mathcal{X}$, 	
	\begin{align}\label{NGH}
			\|x\|\leq  \left(\frac{1}{1-\displaystyle \max_{j \in M, k \in N}|g_k(\tau_j)|}\right)\max\left\{\displaystyle \max_{j \in M^c}|f_j(x) |, \displaystyle \max_{k \in N^c}|g_k(x)|\right\}.
	\end{align}
	In particular, if $x$ is supported on $M$ in the expansion using basis $(\{f_j\}_{j=1}^n, \{\tau_j\}_{j=1}^n)$	 and $x$ is supported on $N$ in the expansion using basis $(\{g_k\}_{k=1}^n, \{\omega_k\}_{k=1}^n)$, then $x=0$. 
\end{theorem}
\begin{proof}
Given $S\subseteq \{1, \dots, n\}$, define 

\begin{align*}
	&P_Sx\coloneqq \sum_{j\in S}f_j(x)\tau_j, \quad \forall x \in \mathcal{X}, \\
	&\|x\|_{S, f}\coloneqq \displaystyle \max_{j \in S}|f_j(x)|, \quad \forall x \in \mathcal{X},\\
	&\|x\|_{S, g}\coloneqq \displaystyle \max_{j \in S}|g_j(x)|, \quad \forall x \in \mathcal{X}.
\end{align*}
Then 
\begin{align*}
	\|P_Sx\|=\left\|\sum_{j\in S}f_j(x)\tau_j\right\|=\displaystyle \max_{j \in S}|f_j(x)|=	\|x\|_{S, f}, \quad \forall x \in \mathcal{X}.
\end{align*}
Define $V:\mathcal{X}\ni x \mapsto \sum_{k=1}^{n}g_k(x)\tau_k \in \mathcal{X}$. Then $V$ is an invertible isometry. Using $V$ we get 
\begin{align*}
	\|P_SVx\|&=\left\|\sum_{j\in S}f_j(Vx)\tau_j\right\|=\left\|\sum_{j\in S}f_j\left(\sum_{k=1}^{n}g_k(x)\tau_k\right)\tau_j\right\|=\left\|\sum_{j\in S}\sum_{k=1}^{n}g_k(x)f_j(\tau_k)\tau_j\right\|\\
	&=\left\|\sum_{j\in S}g_j(x)\tau_j\right\|=\displaystyle \max_{j \in S}|g_j(x)|=	\|x\|_{S, g}, \quad \forall x \in \mathcal{X}.
\end{align*}
For $x \in \mathcal{X}$, we  now find 
\begin{align*}
	\|P_NVP_Mx\|&=\left\|\sum_{k\in N}f_k(VP_Mx)\tau_k\right\|=\max_{k \in N}|f_k(VP_Mx)|\\
	&=\max_{k \in N}\left|(f_kV) \left(\sum_{j\in M}f_j(x)\tau_j\right)\right|=\max_{k \in N}\left| \sum_{j\in M}f_j(x)f_k(V\tau_j)\right|\\
	&=\max_{k \in N}\left| \sum_{j\in M}f_j(x)f_k\left(\sum_{r=1}^{n}g_r(\tau_j)\tau_r\right)\right|
	=\max_{k \in N}\left| \sum_{j\in M}f_j(x)\sum_{r=1}^{n}g_r(\tau_j)f_k(\tau_r)\right|\\
	&=\max_{k \in N}\left| \sum_{j\in M}f_j(x)g_k(\tau_j)\right|\leq  \max_{k \in N}\max_{j \in M}|f_j(x)g_k(\tau_j)|\\
	&\leq \left(\max_{j \in M, k\in N}|g_k(\tau_j)|\right)\max_{j \in M}|f_j(x)|\leq \left(\max_{j \in M, k\in N}|g_k(\tau_j)|\right)\|x\|.
\end{align*}
So we have 
\begin{align}\label{FPVP}
	\|P_NVP_M\|\leq \displaystyle \max_{j \in M, k\in N}|g_k(\tau_j)|.
\end{align}
Now let $y \in \mathcal{X}$ be such that $\{j \in \{1, \dots, n\}: f_j(y)\neq 0\} \subseteq M.$
Then 
\begin{align*}
	\|P_NVy\|=\|P_NVP_My\|\leq \|P_NVP_M\|\|y\|
\end{align*}
 and 
\begin{align*}
	\|y\|_{N^c, g}&=\|P_{N^c}Vy\|=\|Vy-P_NVy\|\\
	&\geq \|Vy\|-\|P_NVy\|=\|y\|-\|P_NVy\| \\
	&\geq \|y\|-\|P_NVP_M\|\|y\|.
\end{align*}
Therefore 
\begin{align}\label{INP}
	\|y\|_{N^c, g}\geq (1-\|P_NVP_M\|)\|y\|.
\end{align}
Let $x \in \mathcal{X}$.   Note that $P_Mx$ satisfies $	\{j \in \{1, \dots, n\}: f_j(P_Mx)\neq 0\} \subseteq M.$ Now using (\ref{FPVP}) and (\ref{INP}) we get 

\begin{align*}
	\|x\|&=\|P_{M^c}x+P_Mx\|\leq \max\{\|P_{M^c}x\|, \|P_Mx\|\}\\
	&\leq \max \left\{\|P_{M^c}x\|, \frac{1}{1-\|P_NVP_M\|}	\|P_Mx\|_{N^c, g}\right\} \\
	&= \max \left\{\|P_{M^c}x\|, \frac{1}{1-\|P_NVP_M\|}	\|P_{N^c}VP_Mx\|\right\} \\
	&=\max \left\{\|P_{M^c}x\|, \frac{1}{1-\|P_NVP_M\|}	\|P_{N^c}V(x-P_{M^c}x)\|\right\} \\
	&\leq \max \left\{\|P_{M^c}x\|, \frac{1}{1-\|P_NVP_M\|}	\|P_{N^c}Vx\|, \frac{1}{1-\|P_NVP_M\|}	\|P_{N^c}VP_{M^c}x\|\right\} \\
	&\leq \max \left\{\|P_{M^c}x\|, \frac{1}{1-\|P_NVP_M\|}	\|P_{N^c}Vx\|, \frac{1}{1-\|P_NVP_M\|}	\|P_{M^c}x\|\right\} \\
	&=\max \left\{\frac{1}{1-\|P_NVP_M\|}	\|P_{M^c}x\|, \frac{1}{1-\|P_NVP_M\|}	\|P_{N^c}Vx\|\right\} \\
	&=\left(\frac{1}{1-\|P_NVP_M\|}\right)\max \left\{\|P_{M^c}x\|, \|P_{N^c}Vx\|\right\} \\
	&=\left(\frac{1}{1-\|P_NVP_M\|}\right)\max \left\{\|x\|_{M^c, f}, \|x\|_{N^c, g}\right\} \\
	&= \left(\frac{1}{1-\|P_NVP_M\|}\right)\max\left\{\displaystyle \max_{j \in M^c}|f_j(x)|, \displaystyle \max_{k \in N^c}|g_k(x)|\right\}\\
	&\leq  \left(\frac{1}{1-\displaystyle \max_{j \in M, k \in N}|g_k(\tau_j)|}\right)\max\left\{\displaystyle \max_{j \in M^c}|f_j(x) |, \displaystyle \max_{k \in N^c}|g_k(x)|\right\}.
\end{align*}
\end{proof}
By interchanging orthonormal bases in Theorem  \ref{NGJ} we  get the following theorem. 
\begin{theorem}
	(\textbf{Non-Archimedean  Ghobber-Jaming  Uncertainty Principle})
Let $\mathcal{X}$  be a finite dimensional non-Archimedean Banach  space over $\mathbb{K}$.	Let $(\{f_j\}_{j=1}^n, \{\tau_j\}_{j=1}^n)$	and $(\{g_k\}_{k=1}^n, \{\omega_k\}_{k=1}^n)$ be orthonormal bases  for $\mathcal{X}$.  If $M,N\subseteq \{1, \dots, n\}$ are such that 
\begin{align*}
	\displaystyle \max_{j \in M, k \in N}|f_j(\omega_k)|<1,
\end{align*}
then for all $x \in \mathcal{X}$, 	
\begin{align*}
	\|x\|\leq  \left(\frac{1}{1-\displaystyle \max_{j \in M, k \in N}|f_j(\omega_k)|}\right)\max\left\{\displaystyle \max_{k \in M^c}|g_k(x) |, \displaystyle \max_{j \in N^c}|f_j(x)|\right\}.
\end{align*}
In particular, if $x$ is supported on $M$ in the expansion using basis $(\{f_j\}_{j=1}^n, \{\tau_j\}_{j=1}^n)$ and $x$ is supported on $N$ in the expansion using basis $(\{g_k\}_{k=1}^n, \{\omega_k\}_{k=1}^n)$, then $x=0$. 
\end{theorem}
We note that Inequality (\ref{NGH}) will not give Inequality (\ref{PGJ})  as p-adic norm is not given using p-adic inner product.

We end the paper with an important open question. Let $\mathcal{H}$ be a finite dimensional Hilbert space. Given an orthonormal basis  $\{\tau_j\}_{j=1}^n$ for $\mathcal{H}$, the \textbf{ (finite) Shannon entropy}  at a point $h \in \mathcal{H}_\tau$ is defined as 
\begin{align*}
	S_\tau (h)\coloneqq - \sum_{j=1}^{n} \left|\left \langle h, \tau_j\right\rangle \right|^2\log \left|\left \langle h, \tau_j\right\rangle \right|^2\geq 0,
\end{align*}
where $\mathcal{H}_\tau\coloneqq \{h \in \mathcal{H}: \|h\|=1, \langle h , \tau_j \rangle \neq 0, 1\leq j \leq n\}$ \cite{DEUTSCH}. In 1983, Deutsch derived following uncertainty principle for Shannon entropy which is fundamental to several developments in Mathematics and   Physics \cite{DEUTSCH}.
\begin{theorem}\cite{DEUTSCH} (\textbf{Deutsch Uncertainty Principle})  \label{DU}
	Let $\{\tau_j\}_{j=1}^n$,  $\{\omega_j\}_{j=1}^n$ be two orthonormal bases for a  finite dimensional Hilbert space $\mathcal{H}$. Then 
	\begin{align}\label{DUP}
		2 \log n \geq S_\tau (h)+S_\omega (h)\geq -2 \log \left(\frac{1+\displaystyle \max_{1\leq j, k \leq n}|\langle\tau_j , \omega_k\rangle|}{2}\right)	\geq 0, \quad \forall h \in \mathcal{H}_\tau.
	\end{align}
\end{theorem}
Let $\{\tau_j\}_{j=1}^n$ be a p-adic orthonormal basis for a p-adic Hilbert space $\mathcal{X}$. Given a p-adic  orthonormal basis  $\{\tau_j\}_{j=1}^n$ for $\mathcal{X}$, the \textbf{ (finite) p-adic Shannon entropy}  at a point $x \in \mathcal{X}_\tau$ is defined as 
\begin{align*}
	S_\tau (x)\coloneqq - \sum_{j=1}^{n} \left|\left \langle x, \tau_j\right\rangle \right|^2\log \left|\left \langle x, \tau_j\right\rangle \right|^2\geq 0,
\end{align*}
where $\mathcal{H}_\tau\coloneqq \{x \in \mathcal{X}: \|x\|=1, \langle x , \tau_j \rangle \neq 0, 1\leq j \leq n\}$. We then have the following problem.
\begin{question}
	What is the p-adic version of Deutsch uncertainty principle (\ref{DUP})?
\end{question}
It has been observed very recently  \cite{KRISHNA2} that using Buzano inequality, (which we state below) one can give a simple proof of Inequality (\ref{DUP}).

\begin{theorem}  \cite{FUJIIKUBO, BUZANO, STEELE} \label{BT}   (\textbf{Buzano Inequality})
	Let $\mathcal{H}$ be a Hilbert space. Then 
	\begin{align}\label{BI}
		|\langle \tau, h\rangle \langle h, \omega\rangle |\leq \frac{\|h\|^2(|\langle \tau, \omega \rangle|+\|\tau\|\|\omega\|)}{2}, \quad \forall h, \tau, \omega \in \mathcal{H}.
	\end{align}
\end{theorem} 
As  Inequality (\ref{BI}) easily gives Inequality (\ref{DUP}), we formulate following problem whose solution will give a p-adic version of Inequality (\ref{DUP}).
\begin{question}
		What is the p-adic version of Buzano Inequality (\ref{BI})?
\end{question}
We note that there is an improvement of Inequality (\ref{DUP}) due to Maassen and Uffink in 1988 \cite{MAASSENUFFINK} motivated from the conjecture by Kraus \cite{KRAUS} made in 1987 (using Riesz-Thorin interpolation), whose p-adic version is also not known. 

\section{Acknowledgments}
This paper has been partially developed when the author attended the workshop “Quantum groups, tensor categories and quantum field theory”, held in University of Oslo, Norway from January 13 to 17, 2025. This event was organized by University of Oslo, Norway and funded by the Norwegian Research Council through the “Quantum Symmetry” project.\\
The author thanks the anonymous reviewer for his/her study and  suggestions, which
improved the article.

 \bibliographystyle{plain}
 \bibliography{reference.bib}

\end{document}